\documentclass[11pt,a4paper]{article}
\usepackage{latexsym}
\usepackage{amsmath}
\usepackage{amssymb}
\usepackage{amscd}
\usepackage{amsthm}
\usepackage[all]{xy}

\newtheorem{thm}{Theorem}[section]

\newtheorem{lem}[thm]{Lemma}
\newtheorem{prop}[thm]{Proposition}

\theoremstyle{definition}
\newtheorem{defn}[thm]{Definition}
\newtheorem{rem}[thm]{Remark}
\newtheorem{exa}[thm]{Example}

\newif\ifpdf \pdftrue
\ifx\pdfoutput\undefined \pdffalse \fi \ifx\pdfoutput\relax
\pdffalse \fi

\usepackage{makeidx}
\makeindex

\begin{document}

\title{Arithmetic invariants of Euclidean lattice}

\author{Shun Tang}

\date{}

\maketitle

\vspace{-10mm}

\hspace{5cm}\hrulefill\hspace{5.5cm} \vspace{5mm}

\textbf{Abstract.} In this paper we study the arithmetic invariants of Euclidean lattice in the context of Arakelov geometry. We regard a Euclidean lattice as a hermitian vector bundle $\bar E$ on ${\rm Spec}(\mathbb{Z})$ and consider two typical arithmetic analogues of the dimension of the space of global sections of a vector bundle on an algebraic curve. One is $$h^0_{\rm Ar}(\bar E):=\log \vert E\cap B_1 \vert$$ where $B_1$ is the unit ball, and the other is $$h^0_{\theta}(\bar{E}):=\log\sum_{v\in E}e^{-\pi\Vert v\Vert^2}$$ where $\sum_{v\in E}e^{-\pi\Vert v\Vert^2}$ is the theta function of $\bar E$. In this paper, we shall prove the following three statements: (i) the fact that one can not reach an absolute Riemann-Roch theorem for $h^0_{\rm Ar}(\bar E)$ is an instance of the Heissenberg uncertainty principle; (ii) the finiteness of equivalence classes in the genus of a positive quadratic form defined over $\mathbb{Z}$ is equivalent to the finiteness of certain isometry classes of hermitian vector bundles on ${\rm Spec}(\mathbb{Z})$, and it can be deduced from a finiteness theorem in Arakelov theory of ${\rm Spec}(\mathbb{Z})$; (iii) for any smooth function $f$ on $\mathbb{R}_{+}$ such that $f>0$ and that $f\circ {\rm exp}$ is a Schwartz function on $\mathbb{R}$, the Mellin transform of $f$ can be written as an integral over the Arakelov divisor class group of ${\rm Spec}(\mathbb{Z})$.



\section{Introduction}
A Euclidean lattice is the data $(V, \Lambda, \Vert\cdot\Vert)$ of some finite dimensional $\mathbb{R}$-vector space $V$, equipped with some Euclidean norm $\Vert\cdot\Vert$, and of some lattice $\Lambda$ in $V$. Equivalently, it is the data $\bar{E}:=(E, \Vert\cdot\Vert)$ of some free $\mathbb{Z}$-module of finite rank $E$, and of some Euclidean norm $\Vert\cdot\Vert$ on the $\mathbb{R}$-vector space $E_\mathbb{R}:=E\otimes \mathbb{R}$. The morphism $$E\to E_\mathbb{R},\quad v\mapsto v\otimes 1$$ is injective whose image is a lattice in $E_\mathbb{R}$.

Let $\bar{E}:=(E, \Vert\cdot\Vert)$ be a Euclidean lattice, it is classically attached some invariants, which depend only of its isomorphism class. For instance, we have

(i) ${\rm rk}E={\rm dim}_\mathbb{R} E_\mathbb{R}\in \mathbb{N}$;

(ii) ${\rm covol}(\bar{E})=m_{\bar{E}}(\Delta)\in \mathbb{R}_{+}^*$, where $m_{\bar{E}}$ denotes the Lebesgue measure on $(E_\mathbb{R},\Vert\cdot\Vert)$ and $\Delta$ denotes a fundamental domain for $E$.

In the context of Arakelov geometry, we call $C:={\rm Spec}(\mathbb{Z})$ an arithmetic curve, its function field is $k(C)=\mathbb{Q}$. Then a Euclidean lattice $\bar{E}$ can be viewed as a hermitian vector bundle over $C$, which can be equipped with more invariants. The one we are most interested in is the analogue of the dimension of the space of global sections of a vector bundle on an algebraic curve. There are two typical ways stemmed from the comparison between number ﬁeld and the function ﬁeld of an algebraic curve deﬁned over ﬁnite ﬁeld to get such arithmetic analogues. 

The first way is to imitate the concept of effectivity of divisors. If $C$ is an algebraic curve over a finite filed of $q$ elements and $D$ is a divisor on $C$, then $$H^0(C, \mathcal{O}(D))=\{f\in K^* \text{ } \vert \text{ } {\rm div}(f)+D\geq 0\}\cup\{0\}$$
is a finite set of cardinality $q^{h^0(C,\mathcal{O}(D))}$, where $h^0(C,\mathcal{O}(D))$ is the dimension of $H^0(C, \mathcal{O}(D))$.

Let us consider $C:={\rm Spec}(\mathbb{Z})$. An Arakelov divisor on ${\rm Spec}\mathbb{Z}$ is a finite sum in the following form 
$$\bar{D}=\sum_{p\in {\rm Spm}\mathbb{Z}}n_p[p]+\lambda[\infty],\qquad \lambda\in \mathbb{R}.$$
For $f\in \mathbb{Q}^*$, the principal Arakelov divisor associated to $f$ is 
$$\widehat{\rm div}(f):=\sum_{p\in {\rm Spm}\mathbb{Z}}{\rm ord}_p(f)[p]-\log\vert f\vert[\infty].$$
For an Arakelov divisor $\bar{D}$, we define its degree as
$$\widehat{\rm deg}(\bar{D}):=\sum_{p\in {\rm Spm}\mathbb{Z}}n_p\log p +\lambda.$$
Note that the degree of a principal Arakelov divisor is $0$. We consider 
\begin{align*}
H^0(C,\mathcal{O}(\bar D)):&=\{f\in \mathbb{Q}^*\text{ }\vert \text{ } \widehat{\rm div}(f)+\bar{D}\geq 0 \}\cup\{0\}\\
&=\{f\in \prod_p p^{-n_p} \text{ } \vert \text{ } \Vert f\Vert_D:=e^{-\lambda}\vert f\vert\leq 1\}
\end{align*}
and define
$$h^0_{\rm Ar}(\mathcal{O}(\bar D)):=\log \vert \prod_p p^{-n_p}\cap B_1 \vert.$$
In general, we may define
$$h^0_{\rm Ar}(\bar E):=\log \vert E\cap B_1 \vert$$
for hermitian vector bundles of higher rank.

The second way is to imitate the construction of the zeta function of an algebraic curve defined over a finite field. If $C$ is an algebraic curve over a finite filed of $q$ elements and $D$ is a divisor on $C$. We denote $N(D)=q^{{\rm deg}(D)}$ by the norm of $D$, then the Zeta function $Z_c(s)$ of $C$ is defined as 
$$Z_c(s)=\sum_{D\geq 0}N(D)^{-s}\qquad ({\rm Re}(s)>1).$$
The norm $N(D)$ only depends on the class $[D]$ of $D$ in the Picard group, so we have
\begin{align*}
Z_c(s)&=\sum_{[D]\in {\rm Pic}(C)}\vert \{D'\in [D]\text{ } \vert \text{ } D'\geq 0\}\vert N[D]^{-s}\\
&=\sum_{[D]\in {\rm Pic}(C)}\frac{q^{h^0(C,\mathcal{O}(D))}-1}{q-1}N[D]^{-s}.
\end{align*}

In \cite{GS}, van der Geer and Schoof introduced a new effectivity concept for Arakelov divisors on ${\rm Spec}(\mathbb{Z})$. Precisely, to an Arakelov divisor $\bar{D}=\sum_{p\in {\rm Spm}\mathbb{Z}}n_p[p]+\lambda[\infty]$, the effectivity of $\bar{D}$ is defined as
$$e(\bar{D}):=\begin{cases}
	{\rm exp}(-\pi\Vert 1\Vert^2_D)  &  \text{if } \prod_p p^{-n_p}\supset \mathbb{Z}  \\
	0  &  \text{otherwise}
\end{cases},$$
then
$$H^0(C,\mathcal{O}(\bar D))=\{f\in \mathbb{Q}^*\text{ }\vert \text{ } e\big(\widehat{\rm div}(f)+\bar{D}\big)>0 \}\cup\{0\}.$$

Let $\zeta(s)=\sum_{(0)\neq J\subset \mathbb{Z}} N(J)^{-s}$ be the Riemann zeta function. Denote $J=\prod_p p^{n_p}$ and $N(\bar{D})=e^{\widehat{\rm deg}(\bar{D})}$. The complete zeta function 
$$Z_{\mathbb Q}(s):=2\pi^{-s/2}\Gamma(s/2)\zeta(s)$$
can be written as 
$$Z_{\mathbb Q}(s)=\int_{\widehat{\rm Pic}(\mathbb{Z})}N([\bar{D}])^{-s}\big(\int_{[\bar{D}]}{\rm exp}(-\pi\Vert 1\Vert^2_D){\rm d}\bar{D}\big){\rm d}[\bar{D}].$$
While
\begin{align*}
	\int_{[\bar{D}]}{\rm exp}(-\pi\Vert 1\Vert^2_D){\rm d}\bar{D} &=\sum_{(0)\neq (f)\subset \prod_p p^{-n_p}}e^{-\pi\Vert 1\Vert^2_{(f)+D}} \\
	&=\frac{1}{w}\big(\sum_{f\in\prod_p p^{-n_p}}e^{-\pi\Vert f\Vert^2_{D}}-1\big).
\end{align*}
where $w=2$ is the number of roots of unity in $\mathbb{Z}$.
This leads to define
$$h^0_{\theta}\big(\mathcal{O}(\bar{D})\big):=\log\sum_{f\in\mathcal{O}(D)}e^{-\pi\Vert f\Vert^2_{D}}.$$
In general, we may define
$$h^0_{\theta}(\bar{E}):=\log\sum_{v\in E}e^{-\pi\Vert v\Vert^2}\in \mathbb{R}_{+}$$
for hermitian vector bundles of higher rank.

Compare these two analogues of $h^0$, the first one $h^0_{\rm Ar}$ looks more natural, but one can not reach an absolute Riemann-Roch theorem for it. The second one $h^0_{\theta}$ fits into an absolute Riemann-Roch theorem
$$h^0_{\theta}(\bar{E})-h^0_{\theta}(\bar{E}^\vee)=\widehat{{\rm deg}}\bar{E}:=\widehat{{\rm deg}}\left(\Lambda^{{\rm rk}E}\bar{E}\right)$$
where $\bar{E}^\vee$ is the dual of $\bar{E}$, since the theta function has functional equation which is highly symmetric.

A natural question is that whether there exists a uniform formalism for $h^0_{\theta}$ and $h^0_{\rm Ar}$. Actually in \cite{Ro}, Roessler introduced a framework from quantum view point to study simultaneously the $\theta$-measure and the classical measure so that the Minkowski problem amounts to the computation of the measure of the entire $E$. In this framework, the fact that one can not reach an absolute Riemann-Roch theorem for $h^0_{\rm Ar}(\bar E)$ is an instance of the Heissenberg uncertainty principle. As mentioned by Roessler in \cite{Ro}, this idea seems to go back to Atiyah, but we couldn't find a proof in any literature so we include one in this paper.

On the other hand, the set ${\rm Bun}_n$ of isometry classes of hermitian vector bundles on ${\rm Spec}(\mathbb{Z})$ is one-to-one correspondent to the double quotient ${\rm GL}_n(\mathbb{Q})\backslash{\rm GL}_n(\mathbb{A}_{\mathbb{Q}})/U$ where $\mathbb{A}_{\mathbb{Q}}$ is the Ad\`{e}le ring of $\mathbb{Q}$ and $U$ is the maximal compact subgroup of ${\rm GL}_n(\mathbb{A}_{\mathbb{Q}})$. The “moduli space” of such bundles is the classical quotient of reduction theory of quadratic forms, functions on ${\rm Bun}_n$ are the same as automorphic forms on ${\rm GL}_n(\mathbb{R})$. In this paper, we transform a positive quadratic form to the hermitian norm on certain vector bundle on ${\rm Spec}(\mathbb{Z})$ and prove that the ﬁniteness of equivalence classes in the genus of this positive quadratic form is equivalent to the finiteness of certain isometry classes of hermitian vector bundles, which can be deduced from a finiteness theorem in Arakelov theory of ${\rm Spec}(\mathbb{Z})$. 

The last result in this paper is to relate the Mellin transform of functions on ${\rm Bun}_1$ to arithmetic invariants of hermitian line bundles. Precisely, we will show that for any smooth function $f$ on $\mathbb{R}_{+}$ such that $f>0$ and that $f\circ {\rm exp}$ is a Schwartz function on $\mathbb{R}$, the Mellin transform of $f$ can be written as an integral over the Arakelov divisor class group of ${\rm Spec}(\mathbb{Z})$.

\textbf{Acknowledgements.} This work is partially supported by NSFC (no. 12171325) and by National Key R$\&$D Program of China No. 2023YFA1009702.

\section{Quantum measure and the Minkowski problem}
For a Euclidean lattice $\bar{E}$, we image that very small particles are located at the lattice points in $E$. We wonder the number of particles inside domain $B_r$. We shall think this problem quantum-mechanically.

Let $P_r(v)$ be the probability of the particle $v$ of being found inside $B_r$, the question is what's the sum of all the $P_r(v)$ for $v\in E$.

We assume that the particle waves all have Gaussian probability distributions
$$\sqrt{t}\cdot{\rm exp}(-\pi t\Vert x-v\Vert^2),$$
then the answer is the following integral
$$\int_{B_r}\sum_{v\in E}\sqrt{t}\cdot{\rm exp}(-\pi t\Vert x-v\Vert^2){\rm d}x.$$

We formalize three different approches to count lattice points in certain domain in terms of measures.

(i) The $\theta$-measure on $E$:
$$\mu_{\theta_t}(v)=\sqrt{t}\cdot{\rm exp}(-\pi t\Vert v\Vert^2).$$

(ii) The quantum measure on $E$:
$$\mu_{Q_{t,r}}(v)=\int_{B_r}\sqrt{t}\cdot{\rm exp}(-\pi t\Vert x-v\Vert^2){\rm d}x.$$

(iii) The classical measure on $E$:
$$\mu_{C_r}(v)=\psi_{B_r}(v)$$
where $\psi_S$ is the characteristic function of the set $S$.

The following two propositions were stated in \cite{Ro}, and they connect the above different approches counting lattice points.

\begin{prop}\label{measure1}
	$\lim_{r\to 0}\frac{\mu_{Q_{t,r}}(v)}{{\rm vol}(B_r)}=\mu_{\theta_t}(v)$.
\end{prop}
\begin{proof}
	By the mean-value theorem we know that 
	\begin{align*}
		\mu_{Q_{t,r}}(v) & =\int_{B_r}\sqrt{t}\cdot{\rm exp}(-\pi t\Vert x-v\Vert^2){\rm d}x \\
		& ={\rm vol}(B_r)\cdot\sqrt{t}\cdot{\rm exp}(-\pi t\Vert x_\xi -v\Vert^2)
	\end{align*}
	with $x_\xi\in B_r$. Therefore, we have
	$$\lim_{r\to 0}\frac{\mu_{Q_{t,r}}(v)}{{\rm vol}(B_r)}=\mu_{\theta_t}(v)=\sqrt{t}\cdot{\rm exp}(-\pi t\Vert v\Vert^2)=\mu_{\theta_t}(v).$$
\end{proof}

\begin{prop}\label{measure2}
	$\lim_{t\to \infty}\mu_{Q_{t,r}}(v)=\mu_{C_r}(v)$.
\end{prop}
\begin{proof}
	Note that the limit of the Gaussian probability distribution as $t\to \infty$ is the Dirac $\delta$-function, so we have
       \begin{displaymath}
		 \lim_{t\to \infty}\int_{B_r}\sqrt{t}\cdot{\rm exp}(-\pi t\Vert x-v\Vert^2){\rm d}x
		=	 \int_{B_r}\delta_v(x){\rm d}x 
		=   \psi_{B_r}(v)= \begin{cases}
			1  &  v\in B_r \\
			0  &  v\notin B_r
		\end{cases}.
	\end{displaymath}
\end{proof}

Denote by $\mathcal{S}(V)$ the Schwartz space of a $\mathbb{R}$-vector space $V$. The Fourier transform
$$\mathcal{F}(f)(\xi):=\int_{E_{\mathbb{R}}}f(x)e^{-2\pi i \xi(x)}{\rm d}x$$
with $f\in \mathcal{S}(E_\mathbb{R})$ and $\xi\in E_\mathbb{R}^\vee$ provides an isomorphism 
$$\mathcal{F}: \mathcal{S}(E_\mathbb{R})\simeq \mathcal{S}(E_\mathbb{R}^\vee)$$
which extends to an isomorphism between corresponding spaces of tempered distributions.

The Poisson formula asserts that the counting measures $\sum_{v\in E}\delta_v$ and $\sum_{\xi\in E^\vee}\delta_\xi$ may be deduced from each other by
$$\mathcal{F}(\sum_{v\in E}\delta_v)=\big({\rm covol}(\bar{E})\big)^{-1}\sum_{\xi\in E^\vee}\delta_\xi.$$
Therefore, for any $f\in \mathcal{S}(E_\mathbb{R})$, we have
$$\sum_{v\in E}f(v)=\big({\rm covol}(\bar{E})\big)^{-1}\sum_{\xi\in E^\vee}\mathcal{F}(f)(\xi).$$

Now we define
$$f_{t,r}(x):=\int_{B_r}\sqrt{t}\cdot{\rm exp}(-\pi t\Vert y-x\Vert^2){\rm d}y.$$
Then 
$$(\mathcal{F}f_{t,r})(\xi)=\int_{B_r}t^{-n}\cdot{\rm exp}(-\pi t^{-1}\Vert y^\vee-\xi\Vert^2){\rm d}y^\vee$$
with $y^\vee=\langle \cdot, y\rangle_{E_\mathbb{R}}$.

Applying the Poisson formula to $f_{t,r}$, we get
\begin{align}
\sum_{v\in E}\int_{B_r}\sqrt{t}\cdot{\rm exp}(-\pi t\Vert y-v\Vert^2){\rm d}y=\big({\rm covol}(\bar{E})\big)^{-1}\sum_{\xi\in E^\vee}\int_{B_r}t^{-n}\cdot{\rm exp}(-\pi t^{-1}\Vert y^\vee-\xi\Vert^2){\rm d}y^\vee.
\end{align}

Divide the formula $(1)$ by ${\rm vol}(B_r)$, ans take $t=1, r\to 0$, we have
$$\sum_{v\in E}e^{-\pi \Vert v\Vert^2}=\big({\rm covol}(\bar{E})\big)^{-1}\sum_{\xi\in E^\vee}e^{-\pi \Vert \xi\Vert^2}$$
i.e. 
$$h^0_{\theta}(\bar{E})-h^0_{\theta}(\bar{E}^\vee)=\widehat{\rm deg}(\bar{E}).$$
This is the absolute Riemann-Roch theorem for $h^0_{\theta}$.

Take $r=1$ and $t\to\infty$, the left hand side of $(1)$ is nothing but $h^0_{\rm Ar}(\bar{E})$, but $h^0_{\rm Ar}(\bar{E}^\vee)$ doesn't appear in the right hand side of $(1)$. Take $r=1$ and $t\to 0$, $h^0_{\rm Ar}(\bar{E}^\vee)$ appears in the right hand side of $(1)$, but one can not reach $h^0_{\rm Ar}(\bar{E})$ in the left hand side of $(1)$. From the physical perspective, if $f_{t,r}$ is viewed as a wave function of particle position, its Fourier transform is a wave function of particle momentum. When $r=1$ and $t$ tends to infinity or to $0$, the number of lattice points of $E$ and of its dual lattice in the unit ball $B_1$ will not appear simultaneously in $(1)$. This means that the position and the momentum of particles can not be determined simultaneously, which is known as the Heisenberg uncertainty principle.

In \cite{Bo}, Bost did more work from physical view point, he gave an explicit relationship between the arithmetic invariants $h^0_{\rm Ar}$ and $h^0_{\theta}$ using the Legendre transform between Lagrangian mechanics and Hamiltonian mechanics.

\section{Geometric interpretation of the ﬁniteness of the class number of positive quadratic form}

\begin{thm}\label{moduli}
	Let $\mathbb{A}_{\mathbb{Q}}$ be the Ad\`{e}le ring of $\mathbb{Q}$. There exists a one-to-one correspondence between the set of isometry classes of hermitian vector bundles on ${\rm Spec}\mathbb{Z}$ of rank $n$ (hence the set of isomorphism classes of Euclidean lattices of rank $n$) and the set of double-cosets 
	$${\rm GL}_n(\mathbb{Q})\backslash{\rm GL}_n(\mathbb{A}_{\mathbb{Q}})/ U$$
	where $U=\prod_{p\in \Sigma_{\mathbb{Q}}^f}{\rm GL}_n(\mathbb{Z}_p)\times {\rm O}_n(\mathbb{R})$ is the maximal compact subgroup of ${\rm GL}_n(\mathbb{A}_{\mathbb{Q}})$.
\end{thm}
\begin{proof}
	For $g=\prod_{p\in \Sigma_{\mathbb{Q}}^f} g_p\times g_\infty\in {\rm GL}_n(\mathbb{A}_{\mathbb{Q}})$ and hermitian vector bundle $\bar{E}$, $$\{F_p:=g_p\cdot E_p\}$$ deduces a vector bundle $F$ since $g_p\in {\rm GL}_n(\mathbb{Z}_p)$ for almost all $p$. We may equip $F$ with a hermitian metric by setting
	$$\langle x, y\rangle_F=\langle g_\infty^{-1}\cdot x, g_\infty^{-1}\cdot y\rangle_E.$$
	So we get an action of ${\rm GL}_n(\mathbb{A}_{\mathbb{Q}})$ on the set of hermitian vector bundles on ${\rm Spec}\mathbb{Z}$ of rank $n$. Note that the element of ${\rm GL}_n(\mathbb{Q})$ gives isometry and $U$ is the stabalizer, so we are done.
\end{proof}

\subsection{Moduli of ${\rm GL}_n(\mathbb{Q})\backslash{\rm GL}_n(\mathbb{A}_{\mathbb{Q}})/ U$ and its subset}

Let $G$ be a linear algebraic group over $\mathbb{Q}$, and let $X$ be an affine variety over $\mathbb{Q}$ equipped with a $G$-action. Two rational points $x, y\in X(\mathbb{Q})$ is $G(\mathbb{Z})$-equivalent if there exists $g\in G(\mathbb{Z})$ s.t. $y=g\circ x$.
Similarly, we may define $G(\mathbb{Q})$-equivalence and $G(\mathbb{Z}_p)$-equivalence.

\begin{defn}\label{local-global}
	For $x\in X(\mathbb{Q})$, we define
	$${\rm genus}(x):=\{y\in X(\mathbb{Q})\text{ }\vert\text{ } y\sim_{G(\mathbb{Q})} x \text{ and } y\sim_{G(\mathbb{Z}_p)}x \text{ for all } p\in \Sigma_{\mathbb{Q}}^f\}$$
	and
	$$f_G(x)=\vert \{{\rm genus}(x)/{\sim_{G(\mathbb{Z})}}\}\vert.$$
	If $f_G(x)=1$, we say that $(X, G)$ satisfies the local-global principle.
\end{defn}

\begin{thm}\label{local-global2}
	For $x\in X(\mathbb{Q})$, write $G_x:=\{g\in G\text{ }\vert\text{ }g\circ x=x\}$. Then $f_G(x)$ is equal to the cardinality of the set
	$$\{\bar{g}=G_x(\mathbb{Q})gG_x(\mathbb{A}(\infty))\in G_x(\mathbb{Q})\backslash G_x(\mathbb{A}_{\mathbb{Q}})/G_x(\mathbb{A}(\infty))\text{ }\vert\text{ } \bar{g}\subset G(\mathbb{Q})G(\mathbb{A}(\infty))\}.$$
\end{thm}
\begin{proof}
	We denote by $S$ the set of $G(\mathbb{Z})$-equivalence classes in ${\rm genus}(x)$, and by $M$ the set of elements in $G_x(\mathbb{Q})\backslash G_x(\mathbb{A}_{\mathbb{Q}})/G_x(\mathbb{A}(\infty))$ which is a subset of  $G(\mathbb{Q})G(\mathbb{A}(\infty))$. We shall establish a one-to-one correspondence between $M$ and $S$.
	
	Let ${\bar g}=G_x(\mathbb{Q})gG_x(\mathbb{A}(\infty))\in M$, there exist $g_{\mathbb{Q}}\in G(\mathbb{Q})$ and $g_{\mathbb{A}(\infty)}\in G(\mathbb{A}(\infty))$ such that $g=g_{\mathbb{Q}}\cdot g_{\mathbb{A}(\infty)}$. Define $y=(g_{\mathbb{Q}})^{-1}\circ x$, then $y$ is a rational point of $X$. And we know that for any $p\in \Sigma_{\mathbb{Q}}^f$, the $p$-part of $g$ satisfies
	$$g_p=g_{\mathbb{Q}}\cdot g_{\mathbb{Z}_p}$$ where $g_{\mathbb{Z}_p}\in G(\mathbb{Z}_p)$ is the $p$-part of $g_{\mathbb{A}(\infty)}$. Then $g_{\mathbb{Q}}=g_p\cdot (g_{\mathbb{Z}_p})^{-1}$ so that $$y=\left(g_{\mathbb{Z}_p}\cdot {g_p}^{-1}\right)\circ x=g_{\mathbb{Z}_p}\circ x$$ holds for any $p\in \Sigma_{\mathbb{Q}}^f$. Therefore, we have $y\in {\rm genus}(x)$.
	
	We now define a map $\phi$ from $M$ to $S$, which maps ${\bar g}$ to the $G(\mathbb{Z})$-equivalence class containing $y$. Firstly, we claim that $\phi$ is well defined. In fact, if 
	$${\bar g}=G_x(\mathbb{Q})gG_x(\mathbb{A}(\infty))={\bar g}=G_x(\mathbb{Q})hG_x(\mathbb{A}(\infty)),$$
	then there exist $t_{\mathbb{Q}}\in G_x(\mathbb{Q})$ and $t_{\mathbb{A}(\infty)}\in G_x(\mathbb{A}(\infty))$ such that $h=t_{\mathbb{Q}}\cdot g\cdot t_{\mathbb{A}(\infty)}$. Write $h=h_{\mathbb{Q}}\cdot h_{\mathbb{A}(\infty)}\in G(\mathbb{Q})G(\mathbb{A}(\infty))$, then 
	$$h=h_{\mathbb{Q}}\cdot h_{\mathbb{A}(\infty)}=(t_{\mathbb{Q}}\cdot g_{\mathbb{Q}})(g_{\mathbb{A}(\infty)}\cdot t_{\mathbb{A}(\infty)})$$
	so that 
	$$s:=(h_{\mathbb{Q}})^{-1}\cdot t_{\mathbb{Q}}\cdot g_{\mathbb{Q}}=h_{\mathbb{A}(\infty)}\cdot (t_{\mathbb{A}(\infty)})^{-1}\cdot (g_{\mathbb{A}(\infty)})^{-1}\in G(\mathbb{Q})\cap G(\mathbb{A}(\infty))=G(\mathbb{Z}).$$
	Now, we have $h_{\mathbb{Q}}=t_{\mathbb{Q}}\cdot g_{\mathbb{Q}}\cdot s^{-1}$ and 
	$$\tilde{y}:=(h_{\mathbb{Q}})^{-1}\circ x=\left(s\cdot (g_{\mathbb{Q}})^{-1} \cdot (t_{\mathbb{Q}})^{-1}\right)\circ x=s\circ \left((g_{\mathbb{Q}})^{-1}\circ x\right)=s\circ y,$$
	hence $\tilde{y}$ and $y$ is $G(\mathbb{Z})$-equivalent.
	
	Secondly, we claim that $\phi$ is surjective. In fact, for $y\in {\rm genus}(x)$, there exist $g_{\mathbb{Q}}\in G(\mathbb{Q})$ and $g_p\in G(\mathbb{Z}_p)$ such that $y=(g_{\mathbb{Q}})^{-1}\circ x=g_p\circ x$ for any $p\in \Sigma_{\mathbb{Q}}^f$.
	Let $h$ be the element in $G(\mathbb{A}(\infty))$ such that $h_{\mathbb{Q}}=(g_{\mathbb{Q}})^{-1}$ and $h_p=g_p$ for any $p\in \Sigma_{\mathbb{Q}}^f$. Then $g_{\mathbb{Q}}\cdot h_p\in G_x(\mathbb{Q}_p)$, and $g:=g_{\mathbb{Q}}\cdot h$ is an element in $G_x(\mathbb{A}_{\mathbb{Q}})$. Define ${\bar g}=G_x(\mathbb{Q})gG_x(\mathbb{A}(\infty))\in M$, $\phi$ maps ${\bar g}$ to the $G(\mathbb{Z})$-equivalence class containing $y$ according to our construction.
	
	At last, $\phi$ is injective. Let ${\bar g}, {\bar h}\in M$ with $g=g_{\mathbb{Q}}\cdot g_{\mathbb{A}(\infty)}$ and $h=h_{\mathbb{Q}}\cdot h_{\mathbb{A}(\infty)}$. Suppose that $\phi({\bar g})=\phi({\bar h})$, then there exists $s\in G(\mathbb{Z})$ such that 
	$$(h_{\mathbb{Q}})^{-1}\circ x=s\circ \left(\left(g_{\mathbb{Q}}\right)^{-1}\circ x\right).$$
	Define $t_{\mathbb{Q}}=h_{\mathbb{Q}}\cdot s\cdot (g_{\mathbb{Q}})^{-1}$ and $t_{\mathbb{A}(\infty)}=(g_{\mathbb{A}(\infty)})^{-1}\cdot s^{-1}\cdot h_{\mathbb{A}(\infty)}$. It is clear that $t_{\mathbb{Q}}\in G_x(\mathbb{Q})$. Since
	$$h_{\mathbb{Z}_p}\circ x=(h_{\mathbb{Q}})^{-1}\circ x \quad \text{and}\quad g_{\mathbb{Z}_p}\circ x=(g_{\mathbb{Q}})^{-1}\circ x$$
	holds for every $p\in \Sigma_{\mathbb{Q}}^f$, it is not difficult to verify that $t_{\mathbb{A}(\infty)}\in G_x(\mathbb{A}(\infty))$. So $h=t_{\mathbb{Q}}\cdot g\cdot t_{\mathbb{A}(\infty)}$ which means ${\bar g}={\bar h}$.
\end{proof}

Let $X\subset\mathbb{A}^{n^2+1}$ be the affine variety defined by 
$$x_{ij}-x_{ji}=0\quad\text{and}\quad {\rm det}(x_{ij})y=1.$$
Then $X(\mathbb{Q})$ is one-to-one correspondent to the set of non-degenerate quadratic forms of $n$-variables over $\mathbb{Q}$. 

Note that $G:={\rm GL}_n$ has an action on $X$ given by 
$$g\circ Q=g^{\rm t}\cdot Q\cdot g\quad\text{for}\quad g\in G, Q\in X.$$
We have the following result.

\begin{prop}\label{correspondence1}
	The local-global principle for $(X,G)$ is equivalent to the local-global principle for quadratic forms. 
\end{prop}
\begin{proof}
	For $x\in X(\mathbb{Q})$, the stabalizer $G_x$ is ${\rm O}_n(x)$. For any $p\in \Sigma_{\mathbb{Q}}^f$, ${\rm GL}_n(\mathbb{Z}_p)$ containes a matrix of determinant $-1$, so the elements in ${\rm O}_n(x)(\mathbb{A}_{\mathbb{Q}})$ can be transfered to  ${\rm SL}_n(\mathbb{A}_{\mathbb{Q}})$. While the special linear algebraic group ${\rm SL}_n$ satisfies the strong approximation theorem, i.e. ${\rm SL}_n(\mathbb{A}_{\mathbb{Q}})={\rm SL}_n(\mathbb{Q}){\rm SL}_n(\mathbb{A}(\infty))$. So we have ${\rm O}_n(x)(\mathbb{A}_{\mathbb{Q}})\subset {\rm GL}_n(\mathbb{Q}){\rm GL}_n(\mathbb{A}(\infty))$. By Theorem~\ref{local-global2}, $f_G(x)$ is equal to the class number of the quadratic form with respect to $x$.
\end{proof}

Assume that $Q$ is a positive quadratic form of $n$-variables over $\mathbb{Q}$. Denote by $G={\rm O}_n(Q)$ the orthogonal group associated to $Q$. Then the set of double-cosets $G(\mathbb{Q})\backslash G(\mathbb{A}_{\mathbb{Q}})/G(\mathbb{A}(\infty))$ is one-to-one correspondent to the set of $\mathbb{Z}$-equivalence classes in ${\rm genus}(Q)$. We use the moduli of ${\rm GL}_n(\mathbb{Q})\backslash{\rm GL}_n(\mathbb{A}_{\mathbb{Q}})/ U$ to construct one-to-one correspondence between the set of $\mathbb{Z}$-equivalence classes in ${\rm genus}(Q)$ and the set of certain isometry classes of hermitian vector bundles on ${\rm Spec}(\mathbb{Z})$ of rank $n$.

\begin{defn}\label{stable}
	A hermitian vector bundle $\bar{E}=(E, \Vert \cdot \Vert)$ on ${\rm Spec}(\mathbb{Z})$ is called $\mathbb{Q}$-stable if $\langle e, e'\rangle\in \mathbb{Q}$ for all $e, e'\in E$.
\end{defn}

\begin{defn}\label{integralbundle}
	A $\mathbb{Q}$-stable hermitian vector bundle $\bar{E}=(E, \Vert \cdot \Vert)$ on ${\rm Spec}(\mathbb{Z})$ is called integral if $\langle e, e'\rangle\subset\mathbb{Q}$ with $e,e'\in E$ generate an integral ideal in $\mathbb{Z}$.
\end{defn}

\begin{defn}\label{Q-equivalent}
	Let $\bar{E}=(E, \Vert \cdot \Vert)$ and $\bar{F}=(F, \Vert \cdot \Vert)$ be two hermitian vector bundles on ${\rm Spec}(\mathbb{Z})$ of the same rank. We call the hermitian metrics on $E$ and on $F$ are $\mathbb{Q}$-equivalent, if there exists an isometry $\bar{E}\cong \bar{L}$ such that the hermitian metrics on $L$ and on $F$ are given by the same hermitian inner product.
\end{defn}

\begin{prop}\label{correspondence2}
	There exists a vector bundle $E$ on ${\rm Spec}(\mathbb{Z})$ such that the hermitian bundle $(E, \Vert \cdot \Vert_Q)$ is $\mathbb{Q}$-stable and integral. Here $\Vert \cdot \Vert_Q$ denotes the metric on $E$ given by the quadratic from $Q$. $G(\mathbb{A}_{\mathbb{Q}})\circ \bar{E}$ provides a family of $\mathbb{Q}$-stable, integral hermitian vector bundles $\{\bar{E_i}\}_{i\in I}$ such that all hermitian metrics are $\mathbb{Q}$-equivalent. The set of isometry classes in $\{\bar{E_i}\}_{i\in I}$ is one-to-one correspondent to the set of double-cosets $G(\mathbb{Q})\backslash G(\mathbb{A}_{\mathbb{Q}})/G(\mathbb{A}(\infty))$. Moreover, $\widehat{\rm deg}(\bar{E_i})$ are constant for $i\in I$.
\end{prop}
\begin{proof}
	At first, we choose a vector bundle $E'$ on ${\rm Spec}(\mathbb{Z})$ and equip it with the hermitian metric associated to $Q$. It is clear that $(E', \Vert \cdot \Vert_Q)$ is $\mathbb{Q}$-stable since $Q$ corresponds to a $n\times n$ metrix over $\mathbb{Q}$. Next, the elements $\langle e, e'\rangle\subset\mathbb{Q}$ with $e,e'\in E'$ generate a fractional ideal $\mathbf{a}$ in $\mathbb{Q}$, we may choose a positive integer $c$ such that $c^2\cdot\mathbf{a}$ is an integral ideal. Define $E=cE'$, then $(E, \Vert \cdot \Vert_Q)$ is $\mathbb{Q}$-stable and integral. According to the construction given in Theorem~\ref{moduli},  $G(\mathbb{A}_{\mathbb{Q}})\circ \bar{E}$ provides a family of hermitian vector bundles $\{\bar{E_i}\}_{i\in I}$. They are $\mathbb{Q}$-stable, integral and all hermitian metrics are $\mathbb{Q}$-equivalent because  $G={\rm O}_n(Q)$ is the orthogonal group associated to $Q$. Note that the stabalizer of ${\bar E}$ under the action of $G(\mathbb{A}_{\mathbb{Q}})$ is $G(\mathbb{A}(\infty))$ and the element of $G(\mathbb{Q})$ gives isometry, then the set of isometry classes in $\{\bar{E_i}\}_{i\in I}$ is one-to-one correspondent to $G(\mathbb{Q})\backslash G(\mathbb{A}_{\mathbb{Q}})/G(\mathbb{A}(\infty))$. To prove the last statement, notice that for any $i,j\in I$, the hermitian metrics on $\Lambda^n {\bar E_i}$ and on $\Lambda^n {\bar E_j}$ are given by the same inner product with respect to ${\rm det}(Q)$. On the other hand, for any place $p\in \Sigma_{\mathbb{Q}}^f$, $(E_i)_p$ and $(E_j)_p$ only differ by an element in $G(\mathbb{Q}_p)$ so that $(\Lambda^nE_i)_p$ and $(\Lambda^nE_j)_p$ only differ by $\pm 1$. This means $\Lambda^n {\bar E_i}=\Lambda^n {\bar E_j}$ and hence $\widehat{\rm deg}(\bar{E_i})=\widehat{\rm deg}(\bar{E_j})$.
\end{proof}

\begin{rem}\label{characteristic}
	For a hermitian vector bundle ${\bar E}$ on ${\rm Spec}(\mathbb{Z})$, we define the Euler-Poincar\'{e} characteristic of ${\bar E}$ as 
	$$\chi({\bar E}):=h^0_\theta({\bar E})-h^0_\theta({\bar E}^\vee),$$
	then $\chi({\bar E})=\widehat{\rm deg}(\bar{E})$ by the absolute Riemann-Roch theorem for $h^0_\theta$.
\end{rem}

\subsection{A finiteness theorem in Arakelov theory of ${\rm Spec}(\mathbb{Z})$}

According to Proposition~\ref{correspondence1}, Proposition~\ref{correspondence2} and Remark~\ref{characteristic}, the finiteness of $\mathbb{Z}$-equivalence classes in ${\rm genus}(Q)$ can be deduced from the following theorem.

\begin{thm}\label{finiteness}
	Let $\{\bar{E_i}\}_{i\in I}$ be a family of $\mathbb{Q}$-stable, integral hermitian vector bundles of rank $n$ and all hermitian metrics on $\{E_i\}_{i\in I}$ are $\mathbb{Q}$-equivalent. If there exist $a, b\in \mathbb{R}$ such that $$a<\chi(\bar{E_i})<b$$ for all $i\in I$, then the set of isomorphism classes in $\{\bar{E_i}\}_{i\in I}$ is finite. 
\end{thm}

\begin{lem}\label{finitevalue1}
	Let $\{\bar{E_i}\}_{i\in I}$ be a family of $\mathbb{Q}$-stable, integral hermitian vector bundles of rank $n$. If there exist $a, b\in \mathbb{R}$ such that $$a<\chi(\bar{E_i})<b$$ for all $i\in I$, then $\{\chi(\bar{E_i})\}_{i\in I}$ only take finitely many values.
\end{lem}
\begin{proof}
	According to Remark~\ref{characteristic}, we only need to prove that $\{\widehat{\rm deg}(\Lambda^n\bar{E_i})\}_{i\in I}$ take finitely many values. The fact that $\bar{E_i}$ is $\mathbb{Q}$-stable and integral implies that $\Lambda^n\bar{E_i}$ is $\mathbb{Q}$-stable and integral, we may suppose that $n=1$. Moreover, since the class number of $\mathbb{Q}$ is $1$, we may suppose that all the $E_i$ are equal to a fixed line bundle $L$. Let $s$ be a non-zero global section of $L$, it is well-known that 
	$$\widehat{\rm deg}(\bar L)=\log([L:s\mathbb{Z}])-\log\Vert s\Vert_L.$$
	Then we only need to show that $\{\log\Vert s\Vert_i\}_{i\in I}$ take finitely many values. In fact, $\bar E_i$ is integral so that $\Vert s\Vert_i\in \mathbb{Z}$, while an open interval of $\mathbb{R}$ can only contains finitely many integers. So we are done.
\end{proof}

\begin{lem}\label{finitevalue2}
	Let assumptions and notations be as in Lemma~\ref{finitevalue1}, then there exists a finite set $\Phi\subset \mathbb{Z}^*$ such that $\Vert E_i\Vert^2_i\cap \Phi\neq \emptyset$ for any $i\in I$.
\end{lem}
\begin{proof}
	By Lemma~\ref{finitevalue1}, we may assume that $\chi(\bar E_i)=c$ are constant. Fix a $n$-dimensional $\mathbb{Q}$-vector space $V$, since the class number of $\mathbb{Q}$ is $1$, we may assume that all $E_i$ have the following form
	$$E_i=\mathbb{Z}x_{i1}+\mathbb{Z}x_{i2}+\cdots+\mathbb{Z}x_{in},$$
	where $x_{i1}, x_{i2}, \ldots, x_{in}$ is a basis of $V$. Since a positive matrix over $\mathbb{Q}$ can always decompose as a product of the transpose of some matrix with rational numbers and itself, we may suppose that the matrics on $E_i$ are all induced by the usual inner product on $V$.
	
	Now, let ${\bar F}$ be a $\mathbb{Q}$-stable, integral hermitian vector bundle of rank $n$ such that $\chi({\bar F})=c$ and that $F$ has the following form 
	$$F=\mathbb{Z}z_{1}+\mathbb{Z}z_{2}+\cdots+\mathbb{Z}z_{n},$$
	where $z_{1}, z_{2}, \ldots, z_{n}$ is a basis of $V$. Suppose that the metric on $F$ is also induced by the usual inner product on $V$. 
	
	Construct a linear map $\phi_i: V\to V$ by setting $\phi_i(z_j)=x_{ij}$ for any $1\leq j\leq n$. So we have $\phi_i(F)=E_i$. Write 
	$$x_{ij}=\phi_i(z_j)=\Sigma_l a_{lj}z_l$$ with $a_{lj}\in \mathbb{Q}.$

	According to the construction of the hermitian metric on $\Lambda^n E_i$ which is induced by the metric on $E_i$ and $\chi({\bar F})=\chi({\bar E_i})=c$, we know that ${\rm det}(\langle z_l, z_j\rangle)$ and ${\rm det}(\langle x_{il}, x_{ij}\rangle)$ only differ by a unit in $\mathbb{Z}$. Therefore $\vert {\rm det}(a_{lj})\vert^2 =1\in \mathbb{Z}$ and ${\rm det}(a_{lj})=\pm 1$. We will show that there exists a uniform bound $B\in \mathbb{Q}_+$ such that every ${\bar E_i}$ contains at least one element whose norm is less or equal to $B$.
	
	Firstly, by \cite[Theorem 103:2]{OM}, there exists a positive number $\gamma$ which only depends on $\mathbb{Q}$ satisfying the following property: for any $n\times n$-matrix $(a_{ij})$ over $\mathbb{Q}$ with ${\rm det}(a_{ij})\in \mathbb{Z}^\times$, there exists a non-zero vector $v=(v_1, \ldots, v_n)\in \mathbb{Z}^n$ such that
    $$\vert (a_{ij})\cdot v^{\rm T}\vert_\infty=\max_{i}\vert a_{i1}v_1+\cdots+a_{in}v_n\vert \leq \gamma.$$

	We apply the above result to the matrix $(a_{lj})$ and write
	$$z^{(i)}=v_1z_1+\cdots+v_nz_n\in F,$$
	then $\Vert \phi_i(z^{(i)})\Vert_i\neq 0$ because $z^{(i)}\neq 0$. Moreover $\phi_i(z^{(i)})\in E_i$ and $\Vert \phi_i(z^{(i)})\Vert^2_i\in \mathbb{Z}$ because $\bar E_i$ is an integral hermitian bundle.
	
	Secondly, write 
	$$\phi_i(z^{(i)})=\Sigma_l \eta_l z_l,\qquad \eta_l=\Sigma_j a_{lj}v_j,$$
	then $\vert \eta_l\vert\leq \gamma$ holds for any $1\leq j\leq n$. One can easily verify that 
	$$\Vert \phi_i(z^{(i)})\Vert^2_i\leq n^2\gamma^2 \max_{l,j}\vert \langle z_l, z_j\rangle\vert.$$
	Define $B=n\gamma \sqrt{\max_{l,j}\vert \langle z_l, z_j\rangle\vert}$, then for any $i\in I$, $\phi_i(z^{(i)})$ is an element in $E_i$ whose norm is less or equal to $B$. This concludes the proof.
\end{proof}

\begin{proof}(of Theorem~\ref{finiteness})
	We prove the statement by induction on $n$. For $n=1$, take isometric linear transforms if necessary, we may assume that $\langle 1_{\mathbb{C}}, 1_{\mathbb{C}}\rangle_i$ are constant. By Lemma~\ref{finitevalue1}, we may suppose that $\chi({\bar E_i})=c$ are constant. Under these assumptions, we shall prove that if ${\bar E_i}$ and ${\bar E_j}$ are algebraically isomorphic to each other, then ${\bar E_i}$ and ${\bar E_j}$ are equal. In fact, if ${\bar E_i}$ is algebraically isomorphic to ${\bar E_j}$, then $E_i$ and $E_j$ differ by a principal ideal $(t)$ with $t\in \mathbb{Q}$. $\chi({\bar E_i})=\chi({\bar E_j})$ implies that $\vert t\vert^2=1$ and hence $t=\pm 1$. This means $E_i=E_j$ and ${\bar E_i}$ and ${\bar E_j}$ are the same hermitian bundles because we have assumed that $\langle 1_{\mathbb{C}}, 1_{\mathbb{C}}\rangle_i=\langle 1_{\mathbb{C}}, 1_{\mathbb{C}}\rangle_j$.
	
	Assume that the statement is true for $n-1$, let us prove that the statement is true for $n$. By Lemma~\ref{finitevalue2}, we have $\cap_{i\in I}\Vert E_i\Vert^2_i\backslash\{0\}\neq \emptyset$. Take an element $\alpha$ in this intersection, then $\alpha\in \mathbb{Z}$. For ${\bar E_i}$, there exists an element $y\in E_i$ such that $\Vert y\Vert^2_i=\alpha$. Denote
	$$F'_i:=\{\alpha x-\langle x,y\rangle_i y \vert x\in E_i\}\quad\quad \mathbf{a}_{iy}:=\{z\in \mathbb{Q} \vert z\cdot y\in E_i\},$$
	then $F'_i$ and $\mathbf{a}_{iy}\cdot y$ are both subbundles of $E_i$. Equip them with the metrics induced by ${\bar E_i}$, we construct a hermitian bundle
	$${\bar E'_i}:=\overline{\mathbf{a}_{iy}\cdot y}\oplus {\bar F'_i}$$
	which is a hermitian subbundle of ${\bar E_i}$. The rank of ${\bar F'_i}$ is less or equal to $n-1$. On the other hand, for any $x\in E_i$, we have
	$$\alpha x=\langle x,y\rangle_i y+(\alpha x-\langle x,y\rangle_i y)\in E'_i.$$
	So $\overline{\alpha E_i}\subset {\bar E'_i}\subset {\bar E_i}$ is a sequence of hermitian subbundles of rank $n$ and hence
	$$\chi(\overline{\alpha E_i})\leq \chi({\bar E'_i})\leq \chi({\bar E_i}),$$
	i.e.
	$$\chi(\overline{\alpha E_i})-\chi(\overline{\mathbf{a}_{iy}\cdot y})\leq \chi({\bar F'_i})\leq \chi({\bar E_i})-\chi(\overline{\mathbf{a}_{iy}\cdot y}).$$
	Notice that $\Lambda^n(\alpha E_i)=\alpha^n \cdot \Lambda^n E_i$, we have $\chi(\overline{\alpha E_i})\geq a-\delta_\alpha$ where $\delta_\alpha$ is a constant depending on $\alpha$. On the other hand, $\mathbf{a}_{iy}$ is a fractional ideal containing $\mathbb{Z}$ and $\Vert y\Vert^2_i=\alpha$ which admits a uniform bound, then $\chi(\overline{\mathbf{a}_{iy}\cdot y})$ is greater or equal to a constant $\eta_\alpha$ and is less or equal to a constant $b_\alpha$ where $\eta_\alpha$ and $b_\alpha$ are both depending on $\alpha$. So we finally have
	$$a-\delta_\alpha-b_\alpha\leq \chi({\bar F'_i})\leq b-\eta_\alpha.$$
	
	Construct ${\bar F'_i}$ similarly for every $i\in I$, then $\{{\bar F'_i}\}_{i\in I}$ is a family of $\mathbb{Q}$-stable, integral hermitian bundles of rank $n-1$. We claim that the metrics on ${\bar F'_i}$ are $\mathbb{Q}$-equivalent. Actually, the metrics on ${\bar E'_i}$ are $\mathbb{Q}$-equivalent, by taking suitable isometric linear transforms, we may fix a vector space $V$ such that all ${\bar E_i}$ are contained in $V$ and the metrics on ${\bar E_i}$ are induced by the usual inner product on $V$. The Witt theorem for hermitian inner product space (cf. \cite{Di}) guarantees that $y$ can be chosen independently of $i$. Then we have an orthogonal decomposition $V=\mathbb{Q}\cdot y \perp U$ where $U$ is a subspace of $V$ of dimension $n-1$ containing $F'_i$. The above orthogonal decomposition induces a hermitian inner product on $U$ which is compatible with the metric on $F'_i$, so the metrics on ${\bar F'_i}$ are $\mathbb{Q}$-equivalent.
	
	By assumption, ${\bar F'_i}$ contains only finitely many isometry classes, denoted by $\{{\bar F_j}\}_{i=1}^m$. Define 
	$${\bar L_j}:=\overline{\mathbb{Z}\cdot y}\oplus {\bar F_j}\qquad (1\leq j\leq m),$$
	then for every ${\bar F'_i}$, there exists an isometry $\phi$ such that 
	$$\phi({\bar E'_i})=\overline{\mathbf{a}_{iy}\cdot y}\oplus {\bar F_j}\supseteq {\bar L_j},$$
	so we have $\phi({\bar E_i})\supseteq \phi({\bar E'_i})\supseteq {\bar L_j}$, and
	$$\langle \phi({\bar E_i}), {\bar L_j}\rangle_V\subseteq \langle \phi({\bar E_i}), \phi({\bar E_i})\rangle_V=\langle {\bar E_i}, {\bar E_i}\rangle_V\subseteq \mathbb{Z}.$$
	Therefore, we get ${\bar L_j}\subseteq \phi({\bar E_i})\subseteq  {\bar L^\#_j}$ where ${L^\#_j}$ is the adjoint bundle of $L_j$ defined as
	$$L_j:=\{x\in V \vert \langle x, e\rangle\in \mathbb{Z}, \forall e\in L_j\}.$$
	Since there are only finitely many hermitian bundles between ${\bar L_j}$ and ${\bar L^\#_j}$, ${\bar E_i}$ contains only finitely many isometry classes. So we are done.
\end{proof}

\section{Integral over the Arakelov divisor class group}
The contents of this section is inspired by \cite[Section 4]{GS} where the complete zeta function 
$$Z_{\mathbb Q}(s):=2\pi^{-s/2}\Gamma(s/2)\zeta(s)$$
was written as an integral over the Arakelov divisor class group of ${\rm Spec}(\mathbb{Z})$ using a new effectivity concept for Arakelov divisors
\begin{align*}
Z_{\mathbb Q}(s)&=\int_{\widehat{\rm Div}(\mathbb{Z})}N(\bar{D})^{-s}{\rm d}\mu_{\rm eff}\\
&=\int_{\widehat{\rm Pic}(\mathbb{Z})}N([\bar{D}])^{-s}\big(\int_{[\bar{D}]}{\rm exp}(-\pi\Vert 1\Vert^2_D){\rm d}\bar{D}\big){\rm d}[\bar{D}].
\end{align*}
The above expression deduces from the fact that the Mellin transform of theta function is the gamma function. The following theorem generalizes the case of theta function to the case of other smooth functions on $\mathbb{R}_{+}$ with positive values.

\begin{thm}\label{integral}
For any smooth function $f$ on $\mathbb{R}_{+}$ such that $f>0$ and that $f\circ {\rm exp}$ is a Schwartz function on $\mathbb{R}$, the Mellin transform of $f$ can be written as an integral over the Arakelov divisor class group of ${\rm Spec}(\mathbb{Z})$.
\end{thm}
\begin{proof}
	If $f(x)$ is a smooth function on $\mathbb{R}_{+}$ such that $f\circ {\rm exp}$ is a Schwartz function on $\mathbb{R}$, it is well known that the Mellin transform $\mathbf{M}(f)(s)$ of $f(x)$ converges within a certain vertical band shaped region $a<{\rm Re}(s)<b$ on the complex plane, where $a$ and $b$ are determined by the decay rate of the function $g(t)=f({\rm exp}(t))$ when $t\to \pm\infty$.
	
	Let $C$ be the maximal value of $f(x)$. For Arakelov divisor $\bar{D}=\sum_{p\in {\rm Spm}\mathbb{Z}}n_p[p]+\lambda[\infty]$, we introduce a new effectivity of $\bar{D}$ as
	$$e(\bar{D}):=\begin{cases}
		\frac{1}{c}f(\Vert 1\Vert^2_D)  &  \text{if } \prod_p p^{-n_p}\supset \mathbb{Z}  \\
		0  &  \text{otherwise}
	\end{cases},$$
	then $0\leq e(\bar{D})\leq 1$ and 
	$$H^0({\rm Spec}(\mathbb{Z}),\mathcal{O}(\bar D))=\{h\in \mathbb{Q}^*\text{ }\vert \text{ } e\big(\widehat{\rm div}(h)+\bar{D}\big)>0 \}\cup\{0\}.$$
	Denote $J=\prod_p p^{n_p}$, $t=e^{-\lambda}$ and $N(\bar{D})=e^{\widehat{\rm deg}(\bar{D})}$. Then
	\begin{align*}
		\int_{\widehat{\rm Div}(\mathbb{Z})}e(\bar{D})N(\bar{D})^{-s}{\rm d}\bar{D}&=\sum_{(0)\neq J\subset \mathbb{Z}} N(J)^{-s}\int_0^{+\infty}t^{s-1}f\left(\frac{1}{c}t^2\right){\rm d}t\\
		&=\sum_{(0)\neq J\subset \mathbb{Z}} N(J)^{-s}\int_0^{+\infty}\frac{1}{2}c^{s/2}x^{s/2-1}f\left(x\right){\rm d}x\\
		&=\frac{1}{2}c^{s/2}\mathbf{M}(f)(s/2)\zeta(s).
	\end{align*}
	
	So we may write
	$$\mathbf{M}(f)(s)=\frac{2}{c^s\zeta(2s)}\int_{\widehat{\rm Div}(\mathbb{Z})}e(\bar{D})N(\bar{D})^{-2s}{\rm d}\bar{D}.$$
	This concluds the proof.
\end{proof}

\begin{exa}\label{deltafunction}
Consider the modular form 
$$\Delta(\tau)=e^{2\pi i \tau}\prod_{n=1}^\infty(1-e^{2\pi i n \tau})^{24}.$$
Let $x$ be positive number, then 
$$\Delta(ix)=e^{-2\pi x}\prod_{n=1}^\infty(1-e^{-2\pi n x})^{24}$$
gives a smooth function on $\mathbb{R}_+$. Actually, $\Delta(ie^t)$ is a Schwartz function on $\mathbb{R}$ because $\Delta$ is analytic, $\Delta(ie^t)\sim e^{-2\pi e^t}$ as $t\to +\infty$, $\Delta(ie^t)\sim e^{12\vert t\vert}\cdot e^{-2\pi e^{\vert t\vert}}$ as $t\to -\infty$ and the derivative of any order of $\Delta(ie^t)$ can be expressed as the product of the derivative of the Delta function and the exponential function, and its decay rate is faster than any polynomial growth. So the Mellin transform of $\Delta(ix)$ converges in a certain vertical band shaped region. According to the Fourier expansion of Delta function, we know that 
$$\Delta(ix)=\sum_{n=1}^\infty\tau(n)e^{-2\pi n x}$$
where $\tau(n)$ is the Ramanujan's Tau function. Therefore
\begin{align*}
\mathbf{M}\left(\Delta(ix)\right)(s)&=\int_0^{+\infty} x^{s-1}\Delta(ix){\rm d}x\\
&=\sum_{n=1}^\infty\tau(n) \int_0^{+\infty} x^{s-1}e^{-2\pi n x}{\rm d}x\\
&=\sum_{n=1}^\infty\tau(n) (2\pi n)^{-s} \int_0^{+\infty} u^{s-1}e^{-u}{\rm d}u\\
&=(2\pi)^{-s}\Gamma(s)L(s,\Delta)
\end{align*}
where 
$$L(s,\Delta)=\sum_{n=1}^\infty\tau(n)n^{-s}$$
is the L-function associated to the Ramanujan's Tau function.

Notice that $\Delta(ix)$ takes the maximal value at some point $0<x_0<1$. For Arakelov divisor $\bar{D}=\sum_{p\in {\rm Spm}\mathbb{Z}}n_p[p]+\lambda[\infty]$, we define the effectivity of $\bar{D}$ associated to $\Delta(ix)$ as
$$e_{\Delta}(\bar{D}):=\begin{cases}
	\frac{1}{\Delta(ix_0)}\Delta(i\Vert 1\Vert^2_D)  &  \text{if } \prod_p p^{-n_p}\supset \mathbb{Z}  \\
	0  &  \text{otherwise}
\end{cases}.$$ 
According to Theorem~\ref{integral}, we have
\begin{align*}
L(s,\Delta)&=\frac{1}{\Gamma(s)}(2\pi)^s \mathbf{M}\left(\Delta(ix)\right)(s)\\
&=\frac{2^{s+1}}{\Gamma(s)\zeta(2s)}\left(\frac{\pi}{\Delta(ix_0)}\right)^s\int_{\widehat{\rm Div}(\mathbb{Z})}e_{\Delta}(\bar{D})N(\bar{D})^{-2s}{\rm d}\bar{D}.
\end{align*}
\end{exa}

\hspace{5cm} \hrulefill\hspace{5.5cm}

Academy for Multidisciplinary Studies \& School of Mathematical Sciences, Capital Normal University, West 3rd Ring North Road 105,
100048 Beijing, P. R. China

E-mail: tangshun@cnu.edu.cn

\end{document}